\documentclass[oneside,leqno,11pt]{amsart}
\usepackage[marginratio=1:1]{geometry}
\usepackage{setspace}
\usepackage{mathpazo,bm}
\usepackage{eucal}
\usepackage{mydefs}

\usepackage[all]{xy}
\newcommand{\sammatrix}[1]{%
  \left[\vcenter{\xymatrix@1@!0@=.6cm@M=0pt{#1}}\right]%
}

\title[Samuel Coskey: Isomorphism and quasi-isomorphism]{The
  classification of torsion-free abelian groups of finite rank up to
  isomorphism and up to quasi-isomorphism}
\author{Samuel Coskey}
\thanks{This is a part of the author's doctoral thesis, which was written
under the supervision of Simon Thomas.}
\thanks{This work was partially supported by NSF grant DMS 0600940}

\begin{document}
\begin{abstract}
  The isomorphism and quasi-isomorphism relations on the $p$-local
  torsion-free abelian groups of rank $n\geq3$ are incomparable with
  respect to Borel reducibility.
\end{abstract}

\maketitle

\section{Introduction}

This paper extends some recent work (by Hjorth, Kechris, Adams and
Thomas) on the complexity of the classification problem for
torsion-free abelian groups of finite rank.  In 1937, Baer solved the
classification problem for torsion-free abelian groups of rank $1$.
The rank $2$ groups resisted satisfactory classification for sixty
years, after which Hjorth \cite{hjorth} used the theory of countable
Borel equivalence relations to prove that the classification problem
for rank $2$ torsion-free abelian groups is genuinely more complex
than that for rank $1$ torsion-free abelian groups.  Building upon the
work of Adams-Kechris \cite{adamskechris}, Thomas \cite{torsionfree}
later proved that the complexity of the classification problem for
torsion-free abelian groups of rank $n$ strictly increases with $n$.

As a stepping stone to this result, Thomas \cite{torsionfree} proved
that the classification problem for torsion-free abelian groups of
rank $n$ up to \emph{quasi-isomorphism} strictly increases in
complexity with $n$.  Here, $A$ and $B$ are said to be
quasi-isomorphic iff $A$ is commensurable with an isomorphic copy of
$B$.  This left open the question of which of the two classification
problems for torsion-free abelian groups of fixed rank $n$ is more
complex: that up to isomorphism or that up to quasi-isomorphism.  In
this paper, we prove that if $n\geq3$ then the two problems have
\emph{incomparable} complexities.

In order to state these results more formally, we must use the
terminology of Borel equivalence relations.  The idea, due to
Friedman-Stanley \cite{friedman} and Hjorth-Kechris
\cite{hjorthkechris}, is that often a ``classification problem'' may
be regarded as an equivalence relation on a standard Borel space.  (A
standard Borel space is a separable, completely metrizable space
equipped only with its $\sigma$-algebra of Borel sets.)  For instance,
any torsion-free abelian group of rank $n$ is isomorphic to a subgroup
of $\QQ^n$.  Hence, any torsion-free abelian group of rank $n$ can be
presented as an element of the standard Borel space $R(n)$ of all
subgroups of $\QQ^n$ that contain a basis of $\QQ^n$.  Studying the
complexity of the classification problem for torsion-free abelian
groups of rank $n$ thus amounts to studying the complexity of the
isomorphism relation $\oiso_n$ on $R(n)$.

The equivalence relations $\oiso_n$ lie in the class of countable
Borel equivalence relations, which we now describe.  An equivalence
relation $E$ on a standard Borel space $X$ is called \emph{Borel} iff
it is a Borel subset of the product space $X\times X$, and it is
called \emph{countable} iff all of its equivalence classes are
countable.  If $\Gamma$ is a countable group of Borel bijections of
$X$, then the corresponding \emph{orbit equivalence relation} defined
by
\[x\mathrel{E_\Gamma}y\iff x,y\textrm{ lie in the same $\Gamma$-orbit}
\]
is easily seen to be a countable Borel equivalence relation.  For
instance, it is easy to verify that subgroups $A,B\leq\QQ^n$ are
isomorphic iff there exists $g\in GL_n(\QQ)$ such that $B=g(A)$.
Hence, $\oiso_n$ is the orbit equivalence relation induced by the
action of $GL_n(\QQ)$ on $R(n)$, and thus it is a countable Borel
equivalence relation.

In fact, by Feldman and Moore \cite{feldmanmoore}, \emph{any}
countable Borel equivalence relation on a standard Borel space $X$
arises as the orbit equivalence relation induced by a Borel action of
a countable group $\Gamma$ on $X$.  We remark that neither the group
$\Gamma$ nor its action is canonically determined by $E_\Gamma$; the
case of $\oiso_n$ is special in the sense that there is a natural
group action which induces it.

We now discuss how to compare the complexity of two equivalence
relations.  If $E,F$ are equivalence relations on the standard Borel
spaces $X,Y$, respectively, then we write $E\leq_BF$ and say that $E$
is \emph{Borel reducible} to $F$ iff there exists a Borel function
$f:X\rightarrow Y$ such that for all $x,x'\in X$,
\[x\mathrel{E}x'\iff f(x)\mathrel{F}f(x').
\]
The relationship $E\leq_BF$ means that elements of $X$ can be
explicitly classified up to $E$ using invariants from the quotient
space $Y/F$, considered with its quotient Borel structure.
Additionally, $E\leq_BF$ implies that structurally, $X/E$ is a simpler
space of invariants than $Y/F$.

Extending this notation in the obvious fashion, we write $E\sim_BF$
when both $E\leq_BF$ and $F\leq_BE$, we write $E<_BF$ when both
$E\leq_BF$ and $E\not\sim_BF$, and we write $E\perp_BF$ when both
$E\not\leq_BF$ and $F\not\leq_BE$.  Returning to torsion-free abelian
groups, we can now state Hjorth's and Thomas's aforementioned theorems
together as:
\[\cong_1\;\;<_B\;\;\cong_2\;\;<_B\;\;\cong_3\;\;<_B\;\;\cdots\;\;<_B\;\;\oiso_n\;\;<_B\;\;\cdots
\]
This was the first naturally occurring example of an infinite
$\leq_B$-chain.  Shortly after this was found, Thomas found an
infinite $\leq_B$-antichain, again consisting of the isomorphism
relations on various spaces of torsion-free abelian groups of finite
rank.  If $p$ is prime, the abelian group $A$ is said to be
\emph{$p$-local} iff it is infinitely $q$-divisible for every prime
$q\neq p$.  Let $R(n,p)$ denote the subspace of $R(n)$ consisting of
just the $p$-local torsion-free abelian groups of rank $n$, and let
$\oiso_{n,p}$ denote the restriction of $\iso_n$ to $R(n,p)$.
Thomas's theorem says that if $n\geq3$ and $p,q$ are distinct primes,
then $\oiso_{n,p}$ is Borel incomparable with $\oiso_{n,q}$.  (This
was later extended to include the case $n=2$ by Hjorth-Thomas
\cite{hjorththomas}.)

In our comparison of isomorphism and quasi-isomorphism, we shall
consider only the $p$-local groups.  It will actually be necessary to
restrict our attention to a slightly smaller space, that of $p$-local
torsion-free abelian groups of fixed \emph{divisible rank}.  Here, the
divisible rank of a (finite rank) torsion-free abelian group $A$ is
defined as the maximum $k$ of the ranks of the divisible quotients of
$A$.  We let $R(n,p,k)$ denote the subspace of $R(n,p)$ consisting of
just the $p$-local torsion-free abelian groups of rank $n$ and of
divisible rank $k$.  Let $\oiso_{n,p}^k$ denote the restriction of
$\iso_n$ to $R(n,p,k)$.

\begin{thma}
  Let $n\geq 3$ and $p$ a prime, and suppose that $k,l<n$ and $k\neq
  l$.  Then $\oiso_{n,p}^k$ is Borel incomparable with
  $\oiso_{n,p}^l$.
\end{thma}

We now turn to a comparison of the isomorphism and quasi-isomorphism
relations on $R(n,p,k)$.  Recall that torsion-free abelian groups
$A,B\leq\QQ^n$ are said to be \emph{quasi-isomorphic} iff $A$ is
commensurable with an isomorphic copy of $B$ (\emph{i.e.}, there
exists $B'\iso B$ such that $A\cap B'$ has finite index in $A$ and in
$B'$).  Let $\oqiso_{n,p}^k$ denote the quasi-isomorphism equivalence
relation on $R(n,p,k)$.  Thomas found the quasi-isomorphism relation
easier to work with in \cite{torsionfree}, for reasons which will
become clear later on in this paper.  However, the next theorem shows
that the classification of ($p$-local) torsion-free abelian groups up
to quasi-isomorphism is not simpler than that up to isomorphism.

\begin{thmb}
  If $1\leq k\leq n-2$, then $\oiso_{n,p}^k$ is Borel incomparable
  with $\oqiso_{n,p}^k$.
\end{thmb}

It should be noted that by Theorem 4.4 of \cite{plocal}, when $k=n-1$
the notion of quasi-isomorphism \emph{coincides} with that of
isomorphism.

It follows easily from Theorems A and B, together with \cite[Theorem
4.7]{plocal} that for $n\geq 3$, the isomorphism and quasi-isomorphism
relations on the space of all \emph{local} (that is, $p$-local for
some $p$) torsion-free abelian groups of rank $n$ are also
incomparable.

\begin{conj}
  For $n\geq 3$, the isomorphism and quasi-isomorphism relations on
  the space $R(n)$ of all torsion-free abelian groups of rank $n$ are
  Borel incomparable.
\end{conj}

This paper is organized as follows.  In the second section, we
introduce ergodic theory and homogeneous spaces.  Our main example is
the $k$-Grassmann space $Gr_k\QQ_p^n$ of $k$-dimensional subspaces of
the $n$-dimensional vector space $\QQ_p^n$ over the $p$-adics.  We
then prove Theorem \ref{thm_affine}, which gives a characterization of
certain action-preserving maps between the $k$-Grassmann spaces.  In
the third section, we state a cocycle superrigidity theorem of Adrian
Ioana, and derive Corollary \ref{cor_rigidity}, which roughly states
that a homomorphism of $SL_n(\ZZ)$-orbits between Grassmann spaces is
a slight perturbation of an action-preserving map.  We then combine
\ref{thm_affine} and \ref{cor_rigidity} to prove the auxiliary result
that for $l\neq n-k$, the orbit equivalence relations induced by the
action of $GL_n(\QQ)$ on $Gr_k\QQ_p^n$ and on $Gr_l\QQ_p^n$ are Borel
incomparable.  In the fourth section, we use the Kurosh-Malcev
completion of torsion-free abelian groups $A\leq\QQ^n$ to establish a
connection (Lemma \ref{lem_containments}) between the space of
$p$-local torsion-free abelian groups and the Grassmann spaces.  In
the last section, we put \ref{thm_affine}, \ref{cor_rigidity}, and
\ref{lem_containments} together to prove Theorems A and B.

\section{Ergodic theory of homogeneous spaces}

In this section, we define the notion of ergodicity of a
measure-preserving action, which plays an essential role in the theory
of countable Borel equivalence relations.  We then consider the case
of countable groups $\Gamma$ acting on homogeneous spaces for compact
$K$ such that $\Gamma\leq K$ (by ``homogeneous,'' we simply mean that
$K$ acts transitively).  As an example, we introduce the Grassmann
space of $k$-dimensional subspaces of $\QQ_p^n$.  The material of this
section is self contained, but we shall see later that there is a
close relationship between Grassmann spaces and spaces of $p$-local
torsion-free abelian groups.

\subsection*{Ergodicity and Borel reductions}

Let $\Gamma$ be a countable group acting in a Borel fashion on the
standard Borel space $X$.  If $X$ carries a Borel probability measure
$\mu$, then we write $\Gamma\actson(X,\mu)$ to indicate that $\Gamma$
acts on $X$ in a $\mu$-preserving fashion.  (When $\mu$ is clear from
context, we often write $\Gamma\actson X$.)  As before, we let
$E_\Gamma$ denote the orbit equivalence relation on $X$ induced by the
action of $\Gamma$.  We say that the action $\Gamma\actson(X,\mu)$ is
\emph{ergodic} iff every $\Gamma$-invariant subset of $X$ is null or
conull for $\mu$.  We shall use the characterization that
$\Gamma\actson(X,\mu)$ is ergodic iff for every $\Gamma$-invariant
function $f:X\rightarrow Y$ into a standard Borel space $Y$, there
exists a conull $A\subset X$ such that $f|_A$ is a constant function.

This characterization leads to an important generalization of
ergodicity which will arise in our arguments.  First, if $E,F$ be
equivalence relations on standard Borel spaces $X,Y$, we define that a
function $f:X\rightarrow Y$ is a \emph{Borel homomorphism} $f$ from
$E$ to $F$ iff
\[x\mathrel{E}y\implies f(x)\mathrel{F}f(y).
\]
(This corresponds to using $Y/F$ as a space of \emph{incomplete}
invariants for the $E$-classification problem on $X$.)  By the last
paragraph, $\Gamma\actson(X,\mu)$ is ergodic iff every Borel
homomorphism from $E_\Gamma$ to $Id_Y$ is constant on a conull set,
where $Id_Y$ represents the \emph{equality} relation on the standard
Borel space $Y$.  More generally, if $F$ is a Borel equivalence
relation on $Y$, then we say $\Gamma\actson (X,\mu)$ is
\emph{$F$-ergodic} iff for every Borel homomorphism $f$ from
$E_\Gamma$ to $F$, there exists a $\mu$-conull subset $A\subset X$
such that $f(A)$ is contained in a single $F$-class.

A countable-to-one Borel homomorphism from $E$ to $F$ is called a
\emph{weak Borel reduction} from $E$ to $F$.  We write $E\leq_B^wF$ if
there exists a weak Borel reduction from $E$ to $F$.  We shall use the
fact that if $E,F$ are countable Borel equivalence relations and $\mu$
is nonatomic, then
\begin{equation}
  \label{eqn_Fergodicity}
  E\textrm{ is }F\textrm{-ergodic}\implies
  E\not\leq_B^wF\implies E\not\leq_BF.
\end{equation}
For the first implication, suppose that $E$ is $F$-ergodic and $f$ is
a weak Borel reduction from $E$ to $F$.  Then there exists a conull
subset $M\subset\dom E$ such that $f(M)$ is contained in a single
$F$-class.  Since $E$ and $F$ are countable and $f$ is
countable-to-one, it follows that $M$ is a countable conull set,
contradicting the fact that $\mu$ is nonatomic.  The second
implication of \eqref{eqn_Fergodicity} is clear from the definitions.

\subsection*{Ergodic components}

If $\Gamma\actson (X,\mu)$ is ergodic and $\Lambda\leq\Gamma$ is a
subgroup of finite index, then there exists a partition
$X=Z_1\sqcup\cdots\sqcup Z_N$ of $X$ into $\Lambda$-invariant subsets
such that for each $i$:
\begin{itemize}
\item $\mu(Z_i)>0$, and
\item $\Lambda\actson(Z_i,\mu_i)$ is ergodic, where $\mu_i$ denotes
  the (normalized) probability measure induced on $Z_i$ by $\mu$.
\end{itemize}
The $\Lambda$-spaces $Z_i$ are called the \emph{ergodic components}
for the action $\Gamma\actson(X,\mu)$.  The set of ergodic components
is determined uniquely up to null sets by the inclusion of $\Lambda$
into $\Gamma$ and the action of $\Gamma$ on $(X,\mu)$.

\subsection*{Homogeneous spaces}

Let $K$ be a compact, second countable group.  If $K$ acts
continuously and transitively on the standard Borel space $X$, then
$X$ is said to be a \emph{homogeneous space} for $K$.  Every
homogeneous space for $K$ is thus in $K$-preserving bijection with
$K/L$ for some closed subgroup $L\leq K$.  Since $K/L$ carries a
unique $K$-invariant probability measure (the projection of the Haar
probability measure on $K$), it follows that $X$ does as well.  Now if
$\Gamma\leq K$ is a countable dense subgroup, then the action of
$\Gamma$ on $K/L$ clearly preserves the Haar measure, and moreover it
is uniquely ergodic with respect to the Haar measure.  (Here, the
action $\Gamma\actson Y$ is said to be \emph{uniquely ergodic} iff
there is a unique $\Gamma$-invariant probability measure on $Y$.)  It
is easy to see that unique ergodicity implies ergodicity, and so
$\Gamma\actson K/L$ is ergodic.

Next, suppose that $\Lambda\leq\Gamma$ is a subgroup of finite index.
Then the ergodic components for the action of $\Lambda$ are precisely
the orbits of $\bar\Lambda$ (the closure in $K$) on $K/L$ and each
ergodic component is again a homogeneous space for the compact group
$K_0=\bar\Lambda$.  If $\Lambda\trianglelefteq\Gamma$ is a
\emph{normal} subgroup of finite index, then $\Gamma$ acts as a
transitive permutation group on the $\bar\Lambda$-orbits, \emph{i.e.},
on the ergodic components for $\Lambda$.  (For proofs of the last few
claims, see \cite[Lemma 2.2]{super}.)

\subsection*{Example: Grassmann spaces}

Let $n$ be a natural number and $p$ a prime.  Denote by $\QQ_p^n$ the
canonical $n$-dimensional vector space over the field of $p$-adic
numbers.  Then the $k$-Grassmann space of $\QQ_p^n$, denoted
$Gr_k\QQ_p^n$, is the set of $k$-dimensional subspaces of $\QQ_p^n$.
Since the compact group $SL_n(\ZZ_p)$ acts transitively on
$Gr_k\QQ_p^n$ \cite[Proposition 6.1]{super} we can view $Gr_k\QQ_p^n$
as a homogeneous $SL_n(\ZZ_p)$-space.  Accordingly, it carries a
corresponding Haar probability measure and the dense subgroup
$SL_n(\ZZ)\leq SL_n(\ZZ_p)$ acts (uniquely) ergodically on
$Gr_k\QQ_p^n$.

We now describe the ``principle congruence components'' of the
$k$-Grassmann space.  Recall that for any natural number $m$, the
principal congruence subgroup $\Gamma_m\trianglelefteq SL_n(\ZZ)$ is
defined by
\[\Gamma_m=\ker[SL_n(\ZZ)\rightarrow SL_n(\ZZ/m\ZZ)]
\]
where the map on the right-hand side is the canonical surjection.  It
is easily seen that the closure in $SL_n(\ZZ_p)$ of $\Gamma_m$ is
exactly $K_m$, where
\[K_m=\ker[SL_n(\ZZ_p)\rightarrow SL_n(\ZZ_p/m\ZZ_p)].
\]
Hence, the ergodic components of $Gr_k\QQ_p^n$ corresponding to the
action of $\Gamma_m$ are precisely the $K_m$-orbits (for example, see
\cite[Lemma 2.2]{super}).  We call these the $m^\textrm{th}$ principle
congruence components of the $k$-Grassmann space.

For example, any $V$ in the $K_{p^t}$ orbit of
$V_0:=\QQ_pe_1\oplus\cdots\oplus\QQ_pe_k$ can be written as the column
space of a matrix $\bracks av$, where $a$ is congruent to the $k\times
k$ identity matrix $I_k$ modulo $p^t$, and $v$ is congruent to $0$
modulo $p^t$.  Since $a$ is clearly invertible, one can use column
operations to suppose that $a=I_k$.  So we have
\begin{equation}
  \label{eqn_Z0}
  (K_{p^t})V_0=\left\{\col\bracks{I_k}{v}:p^t\mid v\right\}
\end{equation}
where $p^t\mid v$ means that for each entry $x$ of $v$, we have that
$x/p^t$ lies in $\ZZ_p$.

Recall that for $n\geq3$, $SL_n(\ZZ)$ has the \emph{congruence
  subgroup property}, meaning that every subgroup of finite index
contains a principle congruence subgroup.  With this, it is easy to
derive the following:

\begin{prop}
  \label{prop_congruence}
  If $n\geq3$, any ergodic component for the action of a subgroup
  $\Gamma\leq SL_n(\ZZ_p)$ of finite index on $Gr_k\QQ_p^n$ contains a
  principle congruence component.
\end{prop}

We close this section with a characterization of the action-preserving
maps between ergodic components of Grassmann spaces.  In what follows,
when $\Gamma\actson X$ and $\Lambda\actson Y$, we shall use the term
\emph{permutation group homomorphism} for a pair $(\phi,f)$ where
$\phi:\Gamma\rightarrow\Lambda$ is a group homomorphism and
$f:X\rightarrow Y$ is a Borel map satisfying $f(\gamma
x)=\phi(\gamma)f(x)$ for all $\gamma\in\Gamma$ and $x\in X$.

\begin{thm}
  \label{thm_affine}
  Let $n\geq 3$ and suppose that $k,l\leq n$.  Let $\Gamma_0,\Gamma_1$
  be subgroups of $SL_n(\ZZ)$ of finite index, $X_0$ an ergodic
  component for the action of $\Gamma_0$ on $Gr_k\QQ_p^n$, and $X_1$
  an ergodic component for the action of $\Gamma_1$ on $Gr_l\QQ_p^n$.
  Suppose that:
  \begin{itemize}
  \item $\phi:\Gamma_0\rightarrow\Gamma_1$ is an isomorphism,
  \item $f:X_0\rightarrow X_1$ is a Borel function, and
  \item $(\phi,f):\Gamma_0\actson X_0\longrightarrow\Gamma_1\actson
    X_1$ is a permutation group homomorphism.
  \end{itemize}
  Then $l=k$ or $l=n-k$, and:
  \begin{enumerate}
  \item If $l=k\neq n-k$ then there exists $h\in GL_n(\QQ)$ such that
    $f$ satisfies $f(x)=hx$ for almost every $x\in X_0$.
  \item If $l=n-k\neq k$ then there exists $h\in GL_n(\QQ)$ such that
    $f$ satisfies $f(x)=hx^\bot$ for almost every $x\in X_0$, where
    $x^\bot$ denotes the orthogonal complement of $x$ with respect to
    the usual dot product.
  \item If $l=k=n/2$ then either the conclusion of (a) holds or the
    conclusion of (b) holds.
  \end{enumerate}
\end{thm}

In the proof, we shall make use of the following well-known result.

\begin{lem}
  \label{lem_margulis}
  Let $n\geq 3$ and $\Gamma_0\leq SL_n(\ZZ)$ be a subgroup of finite
  index.  Let $\phi:\Gamma_0\rightarrow SL_n(\ZZ)$ be an injective
  homomorphism.  Then $\phi$ decomposes as
  $\phi=\epsilon\circ\chi_h\circ(-T)^i$ where:
  \begin{itemize}
  \item $\chi_h(g)=h^{-1}gh$ is conjugation by some $h\in GL_n(\QQ)$,
  \item $-T$ is the inverse-transpose map and $i=0$ or $1$, and
  \item $\epsilon$ is an automorphism of $SL_n(\ZZ)$ satisfying
    $\epsilon(\gamma)=\pm\gamma$.
  \end{itemize}
\end{lem}

\begin{proof}[Proof of Lemma~\ref{lem_margulis}]
  We first suppose that $n$ is odd, so $SL_n(\RR)$ is a simple group.
  By \cite[Theorem IX.5.8]{margulis}, the Zariski closure $H$ of
  $\phi(\Gamma_0)$ in $SL_n(\RR)$ is semisimple.  Let
  $\pi_i:H\rightarrow H_i$ denote the projections of $H$ onto its
  simple factors.  Then $\phi_i(\Gamma_0)$ is Zariski dense in $H_i$,
  and so by the Mostow-Margulis superrigidity theorem (see
  \cite[Theorem 5.1.2]{zimmer}), $\phi_i:=\pi_i\circ\phi$ extends to a
  homomorphism $\Phi_i:SL_n(\RR)\rightarrow H_i$ for each $i$.  Since
  $SL_n(\RR)$ is simple, there is exactly one $\Phi_i$ with infinite
  image, and it follows that the corresponding factor $H_i$ is
  actually $SL_n(\RR)$ itself.  Hence we have that $\phi$ lifts to an
  automorphism $\Phi$ of $SL_n(\RR)$.

  Now, it is well known that the outer automorphism group of
  $SL_n(\RR)$ is comprised just of the identity and the involution
  $-T$.  In other words, we may decompose $\Phi$ as
  $\chi_h\circ(-T)^i$ where $\chi_h$ is conjugation by an element $h$
  of $SL_n(\RR)$ and $i=0$ or $1$.  (When $n$ is even, it may be
  necessary to also include conjugation by a permutation matrix $r$ of
  determinant $-1$.)  Since $\Phi(\Gamma_0)$ is again a lattice of
  $SL_n(\RR)$, we clearly have that $h$ commensurates $SL_n(\ZZ)$.  By
  the proof of \cite[Proposition 6.2.2]{zimmer} (the statement itself
  should be slightly modified to accommodate this case), there exists
  $a\in\RR^*$ such that $ah\in GL_n(\QQ)$.  Of course,
  $\chi_{ah}=\chi_h$, and so the proof is complete in this case.

  In the case that $n$ is even, $\phi$ determines an embedding
  $\bar\phi:\Gamma_0/\mathcal Z(\Gamma_0)\rightarrow PSL_n(\ZZ)$.  One
  may then carry out the above argument inside $PSL_n(\RR)$ to obtain
  that $\bar\phi=\chi_h\circ(-T)^i$.  This map lifts to an
  automorphism $\phi'$ of $SL_n(\ZZ)$, and it is immediate that
  $\phi=\epsilon\circ\phi'$, where $\epsilon$ is as required.
\end{proof}

In the next proposition we shall use the following notation.  For
$V\in Gr_k\QQ_p^n$, let $\stab V$ denote the stabilizer in
$GL_n(\QQ_p)$ of $V$, and if $H\leq GL_n(\QQ_p)$ then let $\stab_H V$
denote the stabilizer in $H$ of $V$.

\begin{prop}
  \label{prop_stab}
  Let $n\geq3$, and let $V\in Gr_k\QQ_p^n$ and $W\in Gr_l\QQ_p^n$.
  Suppose that $K\leq SL_n(\ZZ_p)$ is a subgroup of finite index.  If
  $\stab_KV\subset\stab_KW$ then $l=k$ and $W=V$.
\end{prop}

\begin{proof}
  Since $K$ is Zariski dense in $H=SL_n(\QQ_p)$ (it is an open
  subgroup), we have that $\stab_HV\subset\stab_HW$.  It is well-known
  that $H$ acts \emph{primitively} on each $k$-Grassmann space,
  \emph{i.e.}, $H$ acts transitively on $Gr_k\QQ_p^n$ and the
  stabilizer in $H$ of each point of $Gr_k\QQ_p^n$ is a maximal
  subgroup of $H$.  It follows immediately that we have
  $\stab_HV=\stab_HW$.  Now, it is not hard to see that $V$ is
  uniquely determined by $\stab_HV$ and so $V=W$.
\end{proof}

\begin{proof}[Proof of Theorem \ref{thm_affine}]
  In the notation of Lemma \ref{lem_margulis}, we have that
  $\phi=\epsilon\circ\chi_h\circ(-T)^i$.  Since the center of
  $SL_n(\ZZ)$ acts trivially on $Gr_l\QQ_p^n$, we may suppose that
  $\epsilon$ is the identity map and that $\phi=\chi_h\circ(-T)^i$.
  Having done so, $\phi$ clearly lifts to an automorphism $\Phi$ of
  $GL_n(\QQ_p)$, again defined by the formula $\chi_h\circ(-T)^i$.
  For $i=0,1$, let $K_i$ denote the closure in $SL_n(\ZZ_p)$ of
  $\Gamma_i$, so that $X_i$ is a homogeneous $K_i$-space.  Since
  $\Phi(K_0)$ is a compact group containing $\Gamma_1$, we have that
  $\Phi(K_0)\supset K_1$.  By the same reasoning, we have
  $\Phi^{-1}(K_1)\supset K_0$, and so $\Phi(K_0)=K_1$.  Hence, we may
  define a function $\beta:K_0\rightarrow X_1$ by
  \[\beta(k)=\Phi(k)^{-1}f(kL_0).
  \]
  Then $\beta$ is $\Gamma_0$-invariant, since for $\gamma\in\Gamma_0$
  we have
  \[\beta(\gamma k)=\Phi(\gamma k)^{-1}f(\gamma kL_0)
  =\Phi(k)^{-1}\Phi(\gamma)^{-1}\phi(\gamma)f(kL_0)
  =\Phi(k)^{-1}f(kL_0)
  =\beta(k).
  \]
  Hence, by ergodicity of $\Gamma_0\actson K_0$, there exists $t\in
  K_1$ such that $\beta(k)=tL_1$ for almost all $k\in K_0$.  It
  follows that there exists a conull subset $K_0^*\subset K_0$ such
  that for all $k_0\in K_0^*$
  \[f(k_0L_0)=\Phi(k_0)tL_1
  \]
  We next argue that this implies that if $k_0\in K_0^*$, if
  $x=k_0L_0$ then $\Phi(\stab_{K_0}x)\subset\stab_{K_1}f(x)$.  Indeed,
  suppose that $s\in K_0$ satisfies $sx=x$.  Then choose an element
  $k\in K_0$ such that $kk_0,ksk_0\in K_0^*$ (this is possibile since
  $K_0$ is compact and hence unimodular).  It follows that
  \[f(kx)=\phi(kk_0L_0)=\phi(kk_0)tL_1=\phi(k)\phi(k_0)tL_1=\phi(k)f(x)
  \]
  and
  \[f(kx)=f(ksx)=\phi(ksk_0L_0)=\phi(ksk_0)tL_1=\phi(k)\phi(s)\phi(k_0)tL_1=\phi(k)\phi(s)f(x)
  \]
  and hence $\phi(s)\in\stab_{K_1}f(x)$.

  Finally, since $\Phi$ is either $\chi_h$ or $\chi_h\circ(-T)$, we
  have that either $\stab_{K_1}(h^{-1}x)\subset\stab_{K_1}f(x)$ or
  $\stab_{K_1}(h^{-1}x^\bot)\subset\stab_{K_1}f(x)$.  In the first
  case, we can apply Proposition \ref{prop_stab} to conclude that
  $l=k$ and $f(x)=h^{-1}x$.  In the second case we conclude that
  $l=n-k$ and $f(x)=h^{-1}x^\bot$.
\end{proof}

\section{A superrigidity theorem}

In this section, we describe a recent cocycle superrigidity theorem of
Adrian Ioana (see Chapter 3 of \cite{ioana}).  We then derive a
corollary which will be used in the proofs of our main theorems.
Familiarity with Borel cocycles is needed only to understand the
statement of Ioana's theorem and the proof of the corollary.  While it
is probably also possible to use Zimmer's cocycle superrigidity
theorem in our arguments, doing so would require a more technical
approach.

For $i\in\NN$ let $\Gamma\actson(X_i,\mu_i)$ and
$\rho_i:X_{i+1}\rightarrow X_i$ be a factor map (\emph{i.e.}, a
$\Gamma$-invariant measure-preserving map).  Then the corresponding
\emph{inverse limit} is a $\Gamma$-space $(X,\mu)$ together with
factor maps $\pi_i:X\rightarrow X_i$ satisfying
$\pi_i=\rho_i\circ\pi_{i+1}$ and the usual universal property
associated with inverse limits.

\begin{defn}
  \label{defn_profinite}
  If $(X_i,\mu_i)$ are finite $\Gamma$-spaces (with factor maps as
  above), then the inverse limit $(X,\mu)$ is called a
  \emph{profinite} $\Gamma$-space.
\end{defn}

\begin{thm}[Theorem 3.3.2 of \cite{ioana}]
  \label{thm_ioana}
  Suppose that $\Gamma$ is a countable Kazhdan group, and let
  $(X,\mu)$ be a profinite $\Gamma$-space with corresponding factor
  maps $\pi_i:X\rightarrow X_i$.  Suppose additionally that the action
  $\Gamma\actson(X,\mu)$ is ergodic and free.  If $\alpha:\Gamma\times
  X\rightarrow\Lambda$ is a cocycle into an arbitrary countable group
  $\Lambda$, then there exists $i\in\NN$ such that $\alpha$ is
  cohomologous to a cocycle $\Gamma\times X_i\rightarrow\Lambda$.
  More precisely, there exists $i\in\NN$ and a cocycle
  $\alpha_i:\Gamma\times X_i\rightarrow\Lambda$ such that $\alpha$ is
  cohomologous to the cocycle $\alpha'$ defined by
  $\alpha'(g,x)=\alpha_i(g,\pi_i(x))$.
\end{thm}

This is most useful in the case when $\alpha$ is a cocycle
corresponding to a Borel homomorphism $f$ from $E_\Gamma$ to
$E_\Lambda$, where $E_\Lambda$ is some orbit equivalence relation.  In
this case, the theorem gives hypotheses under which $f$ can be
replaced with an action-preserving map.

\begin{cor}
  \label{cor_rigidity}
  Suppose that $\Gamma$ is a countable Kazhdan group, and let
  $(X,\mu)$ be a free, ergodic profinite $\Gamma$-space.  Let
  $\Lambda$ be a countable group and let $\Lambda\actson Y$ be a free
  action.  Suppose that $f$ is a Borel homomorphism from $E_\Gamma$ to
  $E_\Lambda$.  Then there exists an ergodic component
  $\Gamma_0\actson X_0$ for $\Gamma\actson X$ and a permutation group
  homomorphism $(\phi,f'):\Gamma_0\actson X_0\longrightarrow
  \Lambda\actson Y$ such that for all $x\in X_0$, we have that
  $f'(x)\mathrel{E}_\Lambda f(x)$.
\end{cor}

\begin{proof}
  Let $\alpha:\Gamma\actson X\rightarrow\Lambda$ be the cocycle
  corresponding to $f$.  By Theorem \ref{thm_ioana}, there exists a
  finite factor $(X',\mu')$ of $(X,\mu)$ (denote the projection map by
  $\pi$), and a Borel function $b:X\rightarrow\Lambda$ such that the
  adjusted cocycle
  \[\alpha'(g,x):=b(gx)\alpha(g,x)b(x)^{-1}
  \]
  depends only on $g$ and $\pi(x)$.  Choose any $x_0\in X'$, and let
  $\Gamma_0$ be the stabilizer of $x_0$ in $\Gamma$.  Clearly
  $\Gamma_0\leq\Gamma$ is a subgroup of finite index, and by 3.1.2(ii)
  of \cite{ioana}, $X_0:=\pi^{-1}(x_0)$ is an ergodic component for
  the action of $\Gamma_0$.  Since $\pi$ is constant on $X_0$, the
  restriction of $\alpha'$ to $\Gamma_0\times X_0$ is independent of
  $x\in X_0$.  It follows that $\phi(\gamma):=\alpha'(\gamma,\cdot)$
  defines a homomorphism $\Gamma_0\rightarrow\Lambda$, and then
  letting $f'=bf$ it is easily seen that $(\phi,f')$ satisfies our
  requirements.  (Of course, since $\alpha'$ need only satisfy the
  cocycle identity almost everywhere, one may need to delete a null
  set of $X_0$.)
\end{proof}

We first use this corollary to establish the following Borel
incomparability result.

\begin{defn}
  Fix $n\in\NN$ and $p$ prime.  Let $E_{SL_n\ZZ}^k$ and
  $E_{GL_n\QQ}^k$ denote the orbit equivalence relations induced on
  $Gr_k\QQ_p^n$ by $SL_n(\ZZ)$ and $GL_n(\QQ)$, respectively.
\end{defn}

\begin{thm}
  \label{thm_glnq_separate}
  Suppose that $k,l<n$ and $l$ is neither $k$ nor $n-k$.  Then
  $E_{GL_n\QQ}^k$ is Borel incomparable with $E_{GL_n\QQ}^l$.
\end{thm}

In the next section, we shall see that this theorem implies that the
quasi-isomorphism relations on $R(n,p,k),R(n,p,l)$ are Borel
incomparable.  We prove slightly more:

\begin{thm}
  \label{thm_qiso_ergodic}
  Let $n\geq4$ and $k,l<n$, and suppose that $l$ is neither $k$
  nor $n-k$.  Then $E_{SL_n\ZZ}^k$ (together with the Haar measure on
  $Gr_k\QQ_p^n$) is $E_{GL_n\QQ}^l$-ergodic.
\end{thm}

Theorem \ref{thm_glnq_separate} follows immediately; any Borel
reduction from $E_{GL_n\QQ}^k$ to $E_{GL_n\QQ}^l$ is clearly a weak
Borel reduction from $E_{SL_n\ZZ}^k$ to $E_{GL_n\QQ}^l$, and Theorem
\ref{thm_qiso_ergodic} implies that
$E_{SL_n\ZZ}^k\not\leq_B^wE_{GL_n\QQ}^l$.

\begin{proof}[Proof of Theorem~\ref{thm_qiso_ergodic}]
  Let $f:Gr_k\QQ_p^n\rightarrow Gr_l\QQ_p^n$ be a Borel homomorphism
  from $E_{SL_n\ZZ}^k$ to $E_{GL_n\QQ}^l$, and suppose towards a
  contradiction that $f$ does not map a conull set into a single
  $GL_n(\QQ)$-orbit.  We first reduce the analysis to a situation
  where the hypotheses of Corollary \ref{cor_rigidity} hold.  Since
  $SL_n(\ZZ)$ acts ergodically on $Gr_k\QQ_p^n$, we need only argue
  that both $SL_n(\ZZ)\actson Gr_k\QQ_p^n$ and $GL_n(\QQ)\actson
  Gr_l\QQ_p^n$ are free actions.  While neither action is literally
  free, it will suffice to establish the following statements:

  \begin{claim}[a]
    The action of $PSL_n(\ZZ)$ on $Gr_k\QQ_p^n$ is almost free.
  \end{claim}
  
  \begin{claim}[b]
    The function $f$ maps a (Haar) conull set into the free part of
    the action of $PGL_n(\QQ)$ on $Gr_l\QQ_p^n$.
  \end{claim}

  Here, if the countable group $\Gamma$ acts on a standard Borel space
  $X$, then we let
  \[Fr(\Gamma\actson X):=\{x\in X:1\neq g\in\Gamma\implies gx\neq x\}
  \]
  denote the \emph{free part} of the action of $\Gamma$ on $X$.  If
  $X$ carries a (not necessarily $\Gamma$-invariant) probability
  measure, then we say that $\Gamma\actson X$ is \emph{almost free}
  iff $Fr(\Gamma\actson X)$ is conull.

  \begin{lem}[essentially Lemma 5.1 of \cite{plocal}]
    \label{lem_thomas_free}
    Suppose that $f:Gr_k\QQ_p^n\rightarrow Gr_l\QQ_p^n$ is a Borel
    homomorphism from $E_{SL_n\ZZ}^k$ to $E_{GL_n\QQ}^l$.  Then either
    $f$ maps a conull set into a single $GL_n(\QQ)$-orbit, or there
    exists a conull $X\subset Gr_k\QQ_p^n$ such that
    \[f(X)\subset Fr(PGL_n(\QQ)\actson Gr_l\QQ_p^n).
    \]
  \end{lem}

  It is clear that the lemma establishes Claim (b); Claim (a) also
  follows by applying it in the case when $l=k$ and $f$ is the
  identity map on $Gr_k\QQ_p^n$.  (While Thomas only stated Lemma
  \ref{lem_thomas_free} for the special case when $k=n-1$, his
  argument goes through without change for arbitrary $k,l$.

  \begin{claim}
    We may suppose that there exists an ergodic component
    $\Gamma_0\actson X_0$ for the action $SL_n(\ZZ)\actson
    Gr_k\QQ_p^n$ and a homomorphism $\phi:\Gamma_0\rightarrow
    GL_n(\QQ)$ such that
    \[(\phi,f):\Gamma_0\actson X_0\longrightarrow GL_n(\QQ)\actson
      Gr_l\QQ_p^n
    \]
    is a homomorphism of permutation groups.
  \end{claim}
  
  In the proof, we will in fact produce a $\phi$ such that
  $\phi(\Gamma_0)$ is a subgroup of $SL_n(\ZZ)$ of finite index.

  \begin{claimproof}
  Using Claims (a) and (b) together, it is not difficult to see that
  we may apply Corollary \ref{cor_rigidity} to suppose that there
  exists an ergodic component $\bar\Gamma_0\actson X_0$ for
  $PSL_n(\ZZ)\actson Gr_k\QQ_p^n$ and a homomorphism of permutation
  groups
  \[(\bar\phi,f):\bar\Gamma_0\actson X_0\longrightarrow
  PGL_n(\QQ)\actson Gr_l\QQ_p^n.
  \]
  We wish to lift $\phi$ to a map $\Gamma_0\rightarrow GL_n(\QQ)$,
  where $\Gamma_0$ is the preimage in $SL_n(\ZZ)$ of $\bar\Gamma_0$.

  First, suppose that $\bar\phi$ is not injective.  In this case, by
  Margulis's theorem on normal subgroups (Theorem 8.1.2 of
  \cite{zimmer}), the kernel of $\bar\phi$ has finite index in
  $\bar\Gamma_0$.  Hence, $\bar\phi$ has finite image and so passing
  to an ergodic subcomponent, we can suppose without loss of
  generality that $\bar\phi=1$.  This implies that $f$ is
  $\bar\Gamma_0$-invariant and since $\bar\Gamma_0\actson X_0$ is
  ergodic, $f$ is almost constant.  Hence, in this case $f$ maps a
  conull set into a single $GL_n(\QQ)$-orbit, which is a
  contradiction.

  Next, suppose that $\bar\phi$ is injective.  In this case, we shall
  again make use of Margulis's results.  The next lemma will be used
  in tandem with Lemma \ref{lem_margulis}.

  \begin{lem}
    \label{lem_margulis2}
    If $\Gamma_0\leq SL_n(\ZZ)$ is a finite index subgroup and
    $\phi:\Gamma_0\rightarrow GL_n(\QQ)$ is a homomorphism, then there
    exists a finite index subgroup $\Lambda\leq\Gamma_0$ such that
    $\phi(\Lambda)\leq SL_n(\ZZ)$ is a subgroup of finite index.
  \end{lem}
  
  \begin{proof}
    Since $\Gamma_0$ is Kazhdan (see Theorem 1.5 of \cite{lubotzky}),
    we have that $\Gamma_0':=[\Gamma_0,\Gamma_0]$ is a finite index
    subgroup of $\Gamma_0$ (see Corollary 1.29 of \cite{lubotzky}).
    Now since $GL_n(\QQ)/SL_n(\QQ)\iso\QQ^\times$ is abelian, we have
    that
    \[\phi(\Gamma_0')\leq [GL_n(\QQ),GL_n(\QQ)]\leq SL_n(\QQ).
    \]
    (In fact, the latter $\leq$ is an equality.)  Hence, replacing
    $\Gamma_0$ by $\Gamma_0'$ if necessary, we may suppose without
    loss of generality that $\phi(\Gamma_0)\subset SL_n(\QQ)$.
    Repeating the proof of Lemma \ref{lem_margulis}, after slightly
    adjusting $\phi$ if necessary, we may suppose that it extends to
    an automorphism of $SL_n(\RR)$ (the adjustment is by $\epsilon$,
    in the notation of \ref{lem_margulis}).  It follows that
    $\phi(\Gamma_0)$ is again a lattice of $SL_n(\RR)$.  Since
    $\phi(\Gamma_0)\subset SL_n(\QQ)$, by \cite[IX.4.14]{margulis} we
    have that $\phi(\Gamma_0)$ is commensurable with $SL_n(\ZZ)$, and
    the lemma follows.
  \end{proof}
  
  Although Lemma \ref{lem_margulis2} has been stated so that it will
  be useful later on, for the present circumstances let us note that
  the same proof easily applies to the case of homomorphisms into
  $PGL_n(\QQ)$.  In other words, replacing $X_0$ with a smaller
  ergodic component, we may suppose without loss of generality that
  $\bar\phi(\bar\Gamma_0)$ is a subgroup of $PSL_n(\ZZ)$ of finite
  index.  Moreover, the proof of Lemma \ref{lem_margulis} shows that
  $\bar\phi$ lifts to a homomorphism $\phi:\Gamma_0\rightarrow
  SL_n(\ZZ)$.  Since $\phi$ is a lifting, we easily obtain that
  $f(\gamma x)=\phi(\gamma)f(x)$ for all $\gamma\in\Gamma_0$, which
  completes the proof of the claim.
  \end{claimproof}

  We now wish to maneuver into a situation where we can apply Theorem
  \ref{thm_affine}.

  \begin{claim}
    We may suppose that $f(X_0)\subset X_1$, where $X_1$ is an ergodic
    component for the action of $\phi(\Gamma_0)$ on $Gr_l\QQ_p^n$.
  \end{claim}

  \begin{claimproof}
    Let $Z_1,\ldots,Z_m$ be the ergodic components for the action of
    $\phi(\Gamma_0)$ on $Gr_l\QQ_p^n$.  Now, each $f^{-1}(Z_i)$ is
    $\Gamma_0$-invariant, and since $\Gamma_0\actson X_0$ is ergodic,
    exactly one of the $f^{-1}(Z_i)$ is conull.  Deleting a null
    subset of $X_0$, we may suppose that $f(X_0)\subset Z_i$, as
    desired.
  \end{claimproof}

  Finally, we may apply Theorem \ref{thm_affine} to conclude that
  $l=k$ or $l=n-k$, contradicting our initial hypothesis.  This
  completes the proof of Theorem \ref{thm_qiso_ergodic}.
\end{proof}

\section{Completion of local torsion-free abelian groups}

In this section, we recall some facts surrounding the Kurosh-Malcev
completion of a $p$-local torsion-free abelian group.  The completion
map will allow us to relate the space of $p$-local torsion-free
abelian groups of a fixed divisible rank with the $k$-Grassmann space
$Gr_k\QQ_p^n$.  Using this idea, Thomas essentially showed
\cite[Theorem 4.7]{torsionfree} that the quasi-isomorphism relation
$\oqiso_{n,p}^k$ on $R(n,p,k)$ is bireducible with the orbit
equivalence relation $E_{GL_n\QQ}^k$ induced by the action of
$GL_n(\QQ)$ or $Gr_k\QQ_p^n$.  (Together with Theorem
\ref{thm_glnq_separate}, this implies that $\oqiso_{n,p}^k$ is Borel
incomparable with $\oqiso_{n,p}^l$ for $l\neq k,n-k$.)  We shall use
the completion to establish the following more technical result.

\begin{lem}
  \label{lem_containments}
  There exists an equivalence relation $E_\oiso^k$ on the space
  $Gr_k\QQ_p^n$ satisfying:
  \begin{enumerate}
  \item $E_{SL_n\ZZ}^k\subset E_\oiso^k\subset E_{GL_n\QQ}^k$, and
  \item $E_{\oiso}^k$ is Borel bireducible with the
    isomorphism relation $\oiso_{n,p}^k$ on $R(n,p,k)$.
  \end{enumerate}
\end{lem}

Using this, we can already prove:

\begin{thm}[Theorem A, case 1]
  \label{thm_Acase1}
  Let $n\geq4$ and $k,l<n$, and suppose that $l$ is neither $k$ nor
  $n-k$.  Then $\oiso_{n,p}^k$ is Borel incomparable with
  $\oiso_{n,p}^l$.
\end{thm}

\begin{proof}
  By part (b) of Lemma \ref{lem_containments}, it suffices to prove
  that the relations $E_\oiso^k$ and $E_\oiso^l$ are Borel
  incomparable.  Now, if $f$ is a Borel reduction from $E_\oiso^k$ to
  $E_\oiso^l$, then using part (a) of Lemma \ref{lem_containments} we
  clearly have that $f$ is a weak Borel reduction from $E_{SL_n\ZZ}^k$
  to $E_{GL_n\QQ}^l$.  This contradicts Theorem
  \ref{thm_qiso_ergodic}.
\end{proof}

\subsection*{The completion map}

Following \cite{fuchs}, for $A\in R(n,p)$ we define the
\emph{completion} of $A$ to a $\ZZ_p$-submodule of $\QQ_p^n$ by
\[\Lambda(A):=A\otimes\ZZ_p.
\]
In other words, $\Lambda(A)$ is just the set of all $\ZZ_p$-linear
combinations of elements of $A$ where $A$ is considered as a subset of
$\QQ_p^n$.  The completion map takes values in the standard Borel
space $\mathcal M(n,p)$ of $\ZZ_p$-submodules of $\QQ_p^n$ with
$\ZZ_p$-rank exactly equal to $n$.  In fact, by 93.1 and 93.5 of
\cite{fuchs}, $\Lambda$ is a bijection between $R(n,p)$ and $\mathcal
M(n,p)$.

Now, if $M\in\mathcal M(n,p)$, then by 93.3 of \cite{fuchs} $M$ can be
decomposed as a direct sum
\begin{equation}
  \label{eqn_decompose}
  M=V_M\oplus L
\end{equation}
where $V_M$ is a vector subspace of $\QQ_p^n$ and $L$ is a free
$\ZZ_p$-submodule of $\QQ_p^n$.  The vector subspace $V_M$ is uniquely
determined by $M$.  There are many possible complementary submodules
$L$, but in any case we have $\rank L=n-\dim V$.

By exercise 93.1 of \cite{fuchs}, the dimension of $V_{\Lambda(A)}$ is
precisely the divisible rank of $A$.  Letting $\mathcal M(n,p,k)$
denote the subspace of $\mathcal M(n,p)$ consisting of just those $M$
with $\dim V_M=k$, it follows that $\Lambda$ is a bijection of
$R(n,p,k)$ with $\mathcal M(n,p,k)$.  Since $\Lambda$ is
$GL_n(\QQ)$-invariant, we have that $\oiso_{n,p}^k$ is Borel
equivalent (\emph{i.e.}, isomorphic via a Borel map) to the orbit
equivalence relation induced by the action of $GL_n(\QQ)$ on $\mathcal
M(n,p,k)$.

\subsection*{Decomposition of the space of completed groups}

We now investigate the fibers of the ``vector space part'' map from
$\mathcal M(n,p,k)$ to $Gr_k\QQ_p^n$ defined by $M\mapsto V_M$, where
$V_M$ is as in \eqref{eqn_decompose}.  So, fix $V\in Gr_k\QQ_p^n$ and
let $M\in\mathcal M(n,p,k)$ be an arbitrary module such that $V_M=V$.
If $W$ is any \emph{complementary} subspace of $V$, meaning that
$V\cap W=0$ and $V\oplus W=\QQ_p^n$, then $M$ can always be written
(uniquely) as $V\oplus L$ where $L\leq W$.  Since $L$ is free and
$\rank L=\dim W$, the set $\{M\in\mathcal M(n,p,k):V_M=V\}$ is in
bijection with the lattices of $W$.

\begin{defn}
  Let $K$ be a discrete valuation field and $R$ its ring of integers.
  Then a \emph{lattice} of $K^l$ is a free $R$-submodule of $K^l$ of
  rank $l$.  Equivalently, a lattice of $K^l$ is the $R$-span of $l$
  linearly independent elements of $K^l$.  We denote the set of
  lattices of $K^l$ by $\mathcal L(K^l)$.
\end{defn}

By the discussion in \cite[Section II.1.1]{serre}, since $\QQ$ is a
dense subfield of $\QQ_p$, the map $L\mapsto\ZZ_p\otimes L$ is a
bijection between the lattices of $\QQ^l$ (with respect to the
$p$-adic valuation on $\QQ^*$) and the lattices of $\QQ_p^l$.  Hence,
any lattice of $\QQ_p^l$ may be expressed as the $\ZZ_p$-span of $l$
linearly independent elements of $\QQ^l$.  In particular, there are
only countably many lattices of $\QQ_p^l$ and so the vector space part
map $M\mapsto V_M$ is countable-to-one.

Since any countable-to-one Borel function between standard Borel
spaces admits a Borel section, there clearly exists a Borel bijection
$f:Gr_k\QQ_p^n\times\mathcal L(\QQ_p^{n-k})\rightarrow\mathcal
M(n,p,k)$ satisfying $V_{f(V,L)}=V$ for all $V,L$.  It will be useful
to work with a particular such function.

\begin{defn}
  If $V\in Gr_k\QQ_p^n$, let $V^c$ be the unique complementary
  subspace of $V$ spanned by basis vectors
  $e_{j_1},\ldots,e_{j_{n-k}}$ with the following properties:
  \begin{enumerate}
  \item The $\ZZ_p$-span of
    $(V\cap\ZZ_p^n)\cup(\ZZ_pe_{j_1}\oplus\cdots\oplus\ZZ_pe_{j_{n-k}})$
    is all of $\ZZ_p^n$, and
  \item $\langle j_i\rangle$ is the lexicographically \emph{greatest}
    sequence satisfying (a).
  \end{enumerate}
  We call $V^c$ the \emph{canonical complementary subspace} of $V$.
\end{defn}

We now discuss how to find and identify such a sequence $\langle
j_i\rangle$; in particular we will show that $V^c$ exists.  The key is
that we can write $V$ as the column space of a $n\times k$ matrix $A$
satisfying:
\begin{itemize}
\item Each row of the $k\times k$ identity matrix appears as a row of
  $A$ (call these the \emph{pivot} rows), and
\item Every entry of $A$ is in $\ZZ_p$.
\end{itemize}
(To obtain such a matrix, begin with an arbitrary matrix whose column
space is $V$.  Rescale the first column so that all entries are
$p$-adic integers and at least one entry is $1$.  Use this $1$ to zero
out the other entries in its row.  Repeat for the second column, etc.)
It is easily seen that the sequence $j_i$ of indices of the
\emph{non-pivot} rows of $A$ satisfies (a).  Our requirement that
$\langle j_i\rangle$ is lex-greatest amounts to the more natural
assertion that the sequence of indices of the pivot rows of $A$ is
lex-least.

For example, we have already seen in \eqref{eqn_Z0} that any
$V\in(K_{p^t})V_0$ can be written as the column space of a matrix of
the form $\bracks{I_k}{v}$, where the entries of $v$ are in $\ZZ_p$.
It follows that $V$ has the canonical complementary subspace
$V^c=\QQ_pe_{k+1}\oplus\cdots\oplus\QQ_pe_n$.

Now, for a lattice $L\leq\QQ_p^{n-k}$, let $L[V]$ be the isomorphic
copy of $L$ inside $V^c$ induced by following the obvious map (that
is, the linear map defined by $e_i\mapsto e_{j_i}$).  We define the
\emph{adjoining} of $V$ and $L$ by
\[(V,L):=V\oplus L[V].
\]
It is clear from our construction that adjoining defines a Borel
bijection
\[(\cdot,\cdot):Gr_k\QQ_p^n\times\mathcal L(\QQ_p^{n-k})\rightarrow\mathcal M(n,p,k),
\]
which we shall use to parametrize the elements of $\mathcal
M(n,p,k)$.

\subsection*{Relations on $k$-Grassmann space}

Our approach to Lemma \ref{lem_containments} will be to investigate
the ``canonical'' copy of $Gr_k\QQ_p^n$ in $\mathcal M(n,p,k)$.
Letting $L_0$ denote the standard lattice $\ZZ_p^{n-k}$ of
$\QQ_p^{n-k}$, we put $Y_0:=\{(V,L_0):V\in Gr_k\QQ_p^n\}$.  We first
give the following characterization.

\begin{prop}
  \label{prop_invariant}
  $Y_0$ is precisely the orbit $(SL_n\ZZ_p)M_0$, where $M_0$ is the
  module
  \[M_0=(V_0,L_0)=(\QQ_pe_1\oplus\cdots\oplus\QQ_pe_k)\oplus(\ZZ_pe_{k+1}\oplus\cdots\oplus\ZZ_pe_n)
  \]
\end{prop}

\begin{proof}
  We need only show that for any $g\in SL_n(\ZZ_p)$, we have
  $g(V_0,L_0)\in Y_0$.  Then, since $SL_n(\ZZ_p)$ acts transitively on
  $Gr_k\QQ_p^n$, it follows that $Y_0$ is precisely the orbit
  $(SL_n\ZZ_p)(V_0,L_0)$.

  Suppose first that $g(V_0,L_0)$ is of the form $(V_0,L)$, in other
  words suppose that $gV_0=V_0$.  In this case, $g$ acts on the
  quotient $\QQ_p^n/V_0$ (in the basis represented by
  $e_{k+1},\ldots,e_n$) via its $(n-k)\times(n-k)$ lower right-hand
  corner $g^c$.  It follows easily that $g(V_0,L_0)=(V_0,g^cL_0)$, and
  since the entries of $g^c$ lie in $\ZZ_p$, we clearly have
  $g^cL_0=L_0$.  Hence, $g(V_0,L_0)=(V_0,L_0)$ is an element of $Y_0$.

  Now suppose that $g(V_0,L_0)=(V,L)$ is arbitrary.  It suffices to
  show there exists $g_1\in SL_n(\ZZ_p)$ such that $g_1(V_0,L)=(V,L)$,
  for then, $g_1^{-1}g(V_0,L_0)=(V_0,L)$ and we are in the previous
  case.  Permuting the standard basis if necessary, we can suppose
  that $V=\col\bracks{I_k}{v}$, where the entries of $v$ are in
  $\ZZ_p$.  It follows easily that $g_1:=\sammatrix{I_k&0\\v&I_{n-k}}$
  satisfies our requirements.
\end{proof}

\begin{rem}
  While $Y_0$ is not invariant for the action of $GL_n(\QQ)$, the last
  proposition shows in particular that $Y_0$ is invariant for the
  action of the subgroup $SL_n(\ZZ_{(p)})$.  However, it is not
  difficult to see that the restriction of the orbit equivalence
  relation induced by the action of $GL_n(\QQ)$ on $\mathcal M(n,p,k)$
  to $Y_0$ is not induced by the action of any subgroup of
  $GL_n(\QQ)$.
\end{rem}

\begin{prop}
  \label{prop_complete_section}
  $Y_0$ is a complete Borel section for the orbit equivalence relation
  induced by the action of $GL_n(\QQ)$ on $\mathcal M(n,p,k)$.
\end{prop}

Here, a Borel subset $A\subset X$ is said to be a \emph{complete Borel
  section} for the equivalence relation $E$ on $X$ iff $A$ meets every
$E$-class.  A countable Borel equivalence relation is bireducible with
its restriction to any complete Borel section.  Hence, using the
obvious bijection of $Y_0$ with $Gr_k\QQ_p^n$, we obtain that the
orbit equivalence relation induced by the action of $GL_n(\QQ)$ on
$\mathcal M(n,p,k)$ is bireducible with the following equivalence
relation on $Gr_k\QQ_p^n$:

\begin{defn}
  \label{defn_eoiso}
  For $V,V'\in Gr_k\QQ_p^n$, we define that $V\mathrel{E}_\oiso^k V'$
  iff there exists $g\in GL_n(\QQ)$ such that $(V,L_0)=g(V',L_0)$.
\end{defn}

\begin{proof}[Proof of Proposition \ref{prop_complete_section}]
  We must prove that for every $(V,L)\in\mathcal M(n,p,k)$, there
  exists $g\in GL_n(\QQ)$ and $V'$ such that $g(V,L)=(V',L_0)$.  (Of
  course, $V'$ will be $gV$.)  Permuting the standard basis if
  necessary, we may suppose that $V=\col\bracks{I_k}{v}$, where the
  entries of $v$ are in $\ZZ_p$.  Recall that $L$ has a rational basis
  over $\ZZ_p$, and so there exists $h\in GL_{n\ord-k}(\QQ)$ such that
  $hL=L_0$.  Now, choose a rational matrix $j$ so that the entries of
  $j+hv$ are in $\ZZ_p$, and let \(g=\sammatrix{I_k&0\\j&h}\).  Then
  both $V$ and $gV=\col\bracks{I_k}{j+hv}$ have the canonical
  complementary subspace $V_1=\QQ_pe_{k+1}\oplus\cdots\oplus\QQ_pe_n$.
  One now easily computes that $g(V,L)=(gV,hL)=(gV,L_0)$, as desired.
\end{proof}

\begin{proof}[Proof of Lemma \ref{lem_containments}]
  To see that $E_{SL_n\ZZ}^k\subset E_\oiso^k$, suppose that $g\in
  SL_n(\ZZ)$ and $gV=V'$.  Then by Proposition \ref{prop_invariant},
  we have $g(V,L_0)=(V',L_0)$ and so $V\mathrel{E}_{\oiso}^kV'$.  To
  see that $E_\oiso^k\subset E_{GL_n\QQ}^k$, notice that if $g\in
  GL_n(\QQ)$ and $g(V,L_0)=(V',L_0)$, then $gV=V'$.  We have already
  remarked that the bireducibility of $E_{\oiso}^k$ and
  $\oiso_{n,p}^k$ follows from Proposition
  \ref{prop_complete_section}.
\end{proof}

\section{The proofs of the main theorems}

\begin{thm}[Theorem B, part 1]
  \label{thm_Bpart1}
  If $n\geq3$ and $k\leq n-2$, then
  $\oiso_{n,p}^k\not\leq_B\oqiso_{n,p}^k$.
\end{thm}

By \cite[Theorem 4.3]{plocal}, the quasi-isomorphism relation
$\oqiso_{n,p}^k$ is bireducible with $E_{GL_n\QQ}^k$.  Using this
together with Lemma \ref{lem_containments}, we reduce Theorem
\ref{thm_Bpart1} to the following statement.

\begin{thm}
  \label{thm_Bpart1'}
  If $n\geq3$ and $k\leq n-2$, then $E_\oiso^k\not\leq_BE_{GL_n\QQ}^k$.
\end{thm}

It is worth remarking that this result gives new examples of countable
Borel equivalence relations $E\subset F$ such that $E\not\leq_BF$ (see
\cite{adams}).  Before the proof, we introduce a key invariant on the
space $\mathcal L(\QQ_p^l)$ of lattices of $\QQ_p^l$.

\begin{defn}
  For a lattice $L\in\mathcal L(\QQ_p^l)$, let $A$ be any $l\times l$
  matrix over $\QQ_p$ whose columns form a $\ZZ_p$-basis for $L$.  The
  \emph{type} of $L$, denoted $\type(L)$, is the reduction modulo $l$
  of $\nu_p(\det A)$.  (Here, $\nu_p$ denotes the $p$-adic valuation
  on $\QQ_p^*$.)
\end{defn}

It is easily checked that the type is independent of the choice of the
matrix $A$.  Moreover:

\begin{prop}
  \label{prop_type}
  If $s\in GL_l(\QQ_p)$ then $\type(sL)\equiv\nu_p(\det s)+\type(L)$,
  modulo $l$.
\end{prop}

In particular, the type of a lattice $L$ depends only on its
\emph{class} $\Lambda=\{aL:a\in\QQ_p\}$.

\begin{rem}
  In the case when $l=2$, there is a natural graph structure on the
  set of lattice classes.  Join $\Lambda$ and $\Lambda'$ by an edge
  iff there are $L\in \Lambda$ and $L'\in\Lambda'$ such that $L'$ is a
  maximal proper sublattice of $L$ (or vice versa).  The resulting
  graph is the ($p+1$)-regular tree and the types correspond to the
  colors in a $2$-coloring of the tree.  See the cover of the most
  recent printing of \cite{serre} for a picture in the case that
  $l=p=2$.
\end{rem}

\begin{proof}[Proof of Theorem \ref{thm_Bpart1'}]
  Suppose that $f:Gr_k\QQ_p^n\rightarrow Gr_k\QQ_p^n$ is a Borel
  reduction from $E_\oiso^k$ to $E_{GL_n\QQ}^k$.  Then clearly $f$ is
  a weak Borel reduction from $E_{SL_n\ZZ}^k$ to $E_{GL_n\QQ}^k$.
  Repeating arguments from the proof of Theorem
  \ref{thm_qiso_ergodic}, we may suppose there is an ergodic component
  $\Gamma_0\actson X_0$ for $SL_n(\ZZ)\actson Gr_k\QQ_p^n$ and a
  homomorphism $\phi:\Gamma_0\rightarrow GL_n(\QQ)$ such that
  $(\phi,f):\Gamma_0\actson X_0\longrightarrow GL_n(\QQ)\actson
  Gr_k\QQ_p^n$ is a homomorphism of permutation groups.

  By Lemma \ref{lem_margulis2}, we may replace $\Gamma_0\actson X_0$
  with an ergodic subcomponent to suppose that $\phi(\Gamma_0)\subset
  SL_n(\ZZ)$ is a subgroup of finite index.  Shortly, we shall argue
  that we can suppose that $f(X_0)$ is an ergodic component for the
  action of $\phi(\Gamma_0)$.  However, since our argument is
  sensitive to timing, we must first perform a simplification which
  will make the computations at the end of the proof easier.
  
  By Proposition \ref{prop_congruence}, we may replace $X_0$ with an
  ergodic subcomponent to suppose that $\Gamma_0\actson X_0$ is a
  principle congruence component.  Recall that this means $\Gamma_0$
  is some principle congruence subgroup $\Gamma_{p^t}$ and that $X_0$
  is equal, modulo a null set, to a $K_{p^t}$-orbit.

  \begin{claim}
    We may suppose that the domain $X_0$ of $f$ is equal, modulo a
    null set, to the \emph{particular} ergodic component
    $Z_0=(K_{p^t})V_0$.
  \end{claim}

  Recall that the ergodic component $Z_0$ was described in equation
  \eqref{eqn_Z0}.
  
  \begin{claimproof}
    Recall that $SL_n(\ZZ)$ acts transitively on the $K_{p^t}$-orbits,
    so there exists $\gamma\in SL_n(\ZZ)$ such that $\gamma Z_0=X_0$,
    modulo a null set.  Consider the map $f'(x):=f(\gamma x)$.  By
    Lemma \ref{lem_containments}, we always have
    $x\mathrel{E}_\oiso^k\gamma x$, and so $f'$ a Borel reduction from
    $E_\oiso^k$ to $E_{GL_n\QQ}^k$.  Moreover, it is easily checked
    that $(\phi',f')$ is a homomorphism of permutation groups, where
    $\phi'(g)=\phi(\gamma g\gamma^{-1})$.  Replacing $(\phi,f)$ with
    $(\phi',f')$ establishes the claim.
  \end{claimproof}
  
  \begin{claim}
    We may suppose that $f(x)=x$ for all $x\in X_0$.
  \end{claim}
  
  \begin{claimproof}
    Observe that since the $\gamma$ from the last argument satisfies
    $\gamma\in SL_n(\ZZ)$, we have retained that
    $\phi(\Gamma_0)\subset SL_n(\ZZ)$.  Now, by the ergodicity of
    $\Gamma_0\actson X_0$, we may delete a null subset of $X_0$ to
    suppose that $f(X_0)$ is contained in an ergodic component for the
    action of $\phi(\Gamma_0)$.  By Theorem \ref{thm_affine}, we may
    suppose that there exists $h\in GL_n(\QQ)$ such that $f(x)=hx$ for
    all $x\in X_0$.  But now, since $h\in GL_n(\QQ)$, it follows that
    the identity map on $X_0$ is \emph{also} a Borel reduction from
    $E_\oiso^k$ to $E_{GL_n\QQ}^k$.
  \end{claimproof}
  
  Before proceeding to the final contradiction, we give a brief
  outline.  If indeed the identity function on $X_0$ is a Borel
  reduction from $E_\oiso^k$ to $E_{GL_n\QQ}^k$, then whenever
  $x,gx\in X_0$ and $g\in GL_n(\QQ)$, we will have
  $x\mathrel{E}_\oiso^kgx$.  If we additionally suppose that $x,gx\in
  Z_0$, then as we have observed, $x$ and $gx$ each have the canonical
  complementary subspace $V_1=\QQ_pe_{k+1}\oplus\cdots\oplus\QQ_pe_n$.
  We shall choose the matrix $g$ so that it acts ``nontrivially'' on
  $V_1$ and this will contradict that $x\mathrel{E}_\oiso^kgx$.
  
  Turning to the details, let $g=\diag(1,\ldots,1,p)$, where
  $\diag(d_1,\ldots,d_n)$ denotes the diagonal matrix with
  $a_{ii}=d_i$.  Using equation \eqref{eqn_Z0} one easily checks that
  $gZ_0\subset Z_0$, and since $X_0=Z_0$ modulo a null set, we have
  that $gX_0$ is almost contained in $X_0$.  Together with Lemma
  \ref{lem_thomas_free}, this implies that we can choose $x\in
  Fr(PGL_n(\QQ)\actson Gr_k\QQ_p^n)$ in such a way that $x,gx\in
  X_0\cap Z_0$.  Then $x(E_\oiso^k)gx$, and so Definition
  \ref{defn_eoiso} gives $h\in GL_n(\QQ)$ such that
  \begin{equation}
    \label{eqn_h1}
    h(x,L_0)=(gx,L_0).
  \end{equation}
  Now, $hx=gx$ and since we have chosen $x$ so that it is not fixed by
  any element of $PGL_n(\QQ)\backslash\{1\}$, there exists
  $a\in\QQ_p^*$ such that $h=ag$.  Since $x,gx\in Z_0$, each has the
  canonical complementary subspace
  $V_1=\QQ_pe_{k+1}\oplus\cdots\oplus\QQ_pe_n$.  Since $h$ is
  diagonal, it clearly acts on $V_1$ via its $(n-k)\times(n-k)$ lower
  right-hand corner $h^c$.  It follows that $h(x,L_0)=(hx,h^cL_0)$,
  which (together with \eqref{eqn_h1}) implies that $h^c$ stabilizes
  $L_0$.  But it is readily seen that $\nu_p(\det
  h^c)\equiv1\mod(n-k)$, so Proposition \ref{prop_type} implies that
  $h^c$ does \emph{not} stabilize $L_0$, a contradiction.
\end{proof}

We next attack the hardest case of Theorem A.

\begin{thm}[Theorem A, case 2]
  \label{thm_Acase2}
  Suppose that $n\geq4$ and $2\leq k<n/2$.  Then $\oiso_{n,p}^k$ is
  Borel incomparable with $\oiso_{n,p}^{n\ord-k}$.
\end{thm}

This is a consequence of the following slightly stronger result.

\begin{thm}
  \label{thm_Acase2'}
  Suppose that $n\geq3$, and let $2\leq k\leq n-1$ and $k\neq n/2$.
  Then $E_\oiso^k\not\leq_B^w\oiso_{n,p}^{n\ord-k}$.
\end{thm}

Recall from Section 4 that $\oiso_{n,p}^l$ is Borel bireducible with
the orbit equivalence relation induced by the action of $GL_n(\QQ)$ on
$\mathcal M(n,p,l)$.  In view of the freeness aspects of the proof of
Theorem \ref{thm_qiso_ergodic}, one might expect that it would be
necessary to work with the action of $PGL_n(\QQ)$.  However,
$PGL_n(\QQ)$ does not act on $\mathcal M(n,p,l)$!  Instead, we must
work with the action of $PGL_n(\QQ)$ on the space $\mathcal
M^*(n,p,l)$ of equivalence classes $[M]=\{aM:a\in\QQ_p^*\}$ for
$M\in\mathcal M(n,p,l)$.  (This is indeed a standard Borel space,
since the equivalence relation on $\mathcal M(n,p,k)$ with equivalence
classes $[M]$ is smooth.)  Let $(\oiso_{n,p}^l)^*$ denote the orbit
equivalence relation induced by the action of $GL_n(\QQ)$ on $\mathcal
M^*(n,p,l)$.

\begin{prop}
  The equivalence relation $(\oiso_{n,p}^l)^*$ is Borel bireducible
  with $\oiso_{n,p}^l$.
\end{prop}

\begin{proof}
  Since the completion map $\Lambda$ witnesses that $\oiso_{n,p}^l$ is
  Borel bireducible with the orbit equivalence relation $E$ induced by
  the action of $GL_n(\QQ)$ on $\mathcal M(n,p,l)$, it is enough to
  check that $(\oiso_{n,p}^l)^*$ is Borel bireducible with $E$.
  Clearly, the map $M\mapsto[M]$ is a Borel reduction from $E$ to
  $(\oiso_{n,p}^l)^*$.  Since this map is countable-to-one, it admits
  a Borel section $\sigma$ which is evidently a Borel reduction from
  $(\oiso_{n,p}^l)^*$ to $E$.
\end{proof}

The advantage of working with the space $\mathcal M^*(n,p,l)$ is that
we can use the following variant of Lemma \ref{lem_thomas_free}.

\begin{lem}
  \label{lem_iso_free}
  Suppose that $f:Gr_k\QQ_p^n\rightarrow\mathcal M^*(n,p,l)$ is a weak
  Borel reduction from $E_{SL_n\ZZ}^k$ to $(\oiso_{n,p}^l)^*$.  Then
  there exists a conull subset $X\subset Gr_k\QQ_p^n$ such that
  \[f(X)\subset Fr(PGL_n(\QQ)\actson\mathcal M^*(n,p,l)).
  \]
\end{lem}
  
\begin{proof}
  If $[M]=[M']$, it is clear that $V_M=V_{M'}$ (see the notation of
  \eqref{eqn_decompose}).  Hence the map $[M]\mapsto V_M$ is
  well-defined, and we may consider the function $\bar
  f:Gr_k\QQ_p^n\rightarrow Gr_l\QQ_p^n$ given by $\bar f(x)=V_{f(x)}$.
  Clearly, $\bar f$ is a weak Borel reduction from $E_{SL_n\ZZ}^k$ to
  $E_{GL_n\QQ}^l$.  We claim that there exists a conull subset
  $X\subset Gr_k\QQ_p^n$ such that
  \[\bar f(X)\subset Fr(PGL_n(\QQ)\actson Gr_l\QQ_p^n).
  \]
  If this is not the case, then Theorem \ref{lem_thomas_free} implies
  that there exists a conull subset $X'\subset Gr_l\QQ_p^n$ such that
  $\bar f(X')$ is contained in a single $GL_n(\QQ)$-orbit of
  $Gr_l\QQ_p^n$.  Since $f$ is countable-to-one, it follows that
  $f(X')$ is contained in a countable set, contradicting that the Haar
  measure is nonatomic.

  Now, for $x\in X$, we have that $\bar f(x)\in Fr(PGL_n(\QQ)\actson
  Gr_l\QQ_p^n)$.  This means by definition that $1\neq g\in
  PGL_n(\QQ)$ implies $gV_{f(x)}\neq V_{f(x)}$.  Clearly,
  $gV_{f(x)}=V_{gf(x)}$, and so we have $V_{gf(x)}\neq V_{f(x)}$.  It
  follows that $gf(x)\neq f(x)$, which means that $f(x)\in
  Fr(PGL_n(\QQ)\actson\mathcal M^*(n,p,l))$, as desired.
\end{proof}

We extend the notion of type to $\mathcal M^*(n,p,l)$ by letting
$\type[M]=\type(L)$, where $M=(V,L)$.  This is well-defined, as
$aM=(V,aL)$ and we have already observed that the type of $L$ depends
only on its class.  The following fact is the last we shall need in
the proof of Theorem \ref{thm_Acase2'}.

\begin{prop}
  \label{prop_type_preserve}
  The group $SL_n(\ZZ_p)$ acts in a type-preserving fashion on
  $\mathcal M^*(n,p,l)$.
\end{prop}

\begin{proof}
  We must show that whenever $g\in SL_n(\ZZ_p)$ and $g(V,L)=(V',L')$,
  we have that $\type(L)=\type(L')$.  First suppose that $V=V'=V_0$,
  where $V_0=\QQ_pe_1\oplus\cdots\oplus\QQ_pe_l$.  Then in particular,
  $g$ fixes $V_0$.  Letting $g^c$ denote the $(n-l)\times(n-l)$ lower
  right-hand corner of $g$, we can argue as in the proof of
  Proposition \ref{prop_invariant} that $g(V_0,L)=(V_0,g^cL)$ and so
  $L'=g^cL$.  But $g^c\in GL_{n-l}(\ZZ_p)$, and so Proposition
  \ref{prop_type} implies that $\type(L')=\type(L)$.

  Also as in the proof of Proposition \ref{prop_invariant}, this
  special case can be translated to establish the result in the
  general case.
\end{proof}

\begin{proof}[Proof of Theorem \ref{thm_Acase2'}]
  Suppose that $f:Gr_k\QQ_p^n\rightarrow\mathcal M^*(n,p,n\ord-k)$ is
  a weak Borel reduction from $E_\oiso^k$ to
  $(\oiso_{n,p}^{n\ord-k})^*$.  Then clearly $f$ is a weak Borel
  reduction from $E_{SL_n\ZZ}^k$ to $(\oiso_{n,p}^{n\ord-k})^*$.
  Applying the arguments of Theorem \ref{thm_qiso_ergodic} (and
  substituting Lemma \ref{lem_iso_free} instead of Lemma
  \ref{lem_thomas_free} to define a cocycle), we may suppose that
  there exists an ergodic component $\Gamma_0\actson X_0$ for
  $SL_n(\ZZ)\actson Gr_k\QQ_p^n$, and a homomorphism
  $\phi:\Gamma_0\rightarrow GL_n(\QQ)$ such that
  $(\phi,f):\Gamma_0\actson X_0\longrightarrow
  GL_n(\QQ)\actson\mathcal M^*(n,p,n\ord-k)$ is a homomorphism of
  permutation groups.

  Since the map $[M]\rightarrow V_M$ is $GL_n(\QQ_p)$-preserving, the
  composition $\bar f:x\mapsto V_{f(x)}$ makes
  \[(\phi,\bar f):\Gamma_0\actson Z_0\longrightarrow GL_n(\QQ)\actson
  Gr_{n-k}\QQ_p^n
  \]
  into a homomorphism of permutation groups.  By Lemma
  \ref{lem_margulis2}, we may replace $\Gamma_0\actson X_0$ with an
  ergodic subcomponent to suppose that $\im(\phi)\subset SL_n(\ZZ)$.
  As in the proof of Theorem~\ref{thm_Bpart1'}, we make the following
  simplifications.

  \begin{claim}
    We may suppose that $\Gamma_0=\Gamma_{p^t}$ is a principle
    congruence subgroup and that $X_0$ is equal, modulo a null set, to
    the particular component $Z_0=(K_{p^t})V_0$.\hfill$\dashv$
  \end{claim}

  \begin{claim}
    We may suppose that there exists $h\in GL_n(\QQ)$ such that $\bar
    f(x)=x^\bot$ for all $x\in X_0$.
  \end{claim}

  \begin{claimproof}
    As with last time, the hypotheses of Theorem~\ref{thm_affine} are
    satisfied, and we may suppose that there exists $h\in GL_n(\QQ)$
    such that $\phi(\gamma)=\chi_h\circ(-T)$ and $\bar
    f(x)=h^{-1}x^\bot$.  Hence, replacing $f$ by $hf$ and $\phi$ with
    $-T$, we obtain the desired result.
  \end{claimproof}

  Observe that in the last argument, we have retained the property
  that $(\phi,f)$ is a homomorphism of permutation groups such that
  $\im(\phi)\subset SL_n(\ZZ)$.  For the next claim, recall that we have
  defined that $\type([M])$ is the type of any $L$ such that
  $(V,L)\in[M]$.

  \begin{claim}
    We can suppose that there is a fixed $0\leq t<k$ such that
    $\type(f(x))=t$ for all $x\in X_0$.
  \end{claim}

  \begin{claimproof}
    Let $F(\mathcal M^*(n,p,n\ord-k))$ denote the standard Borel space
    of closed subsets of $\mathcal M^*(n,p,n\ord-k)$ (where $\mathcal
    M(n,p,n\ord-k)$ is considered with the quotient topology induced
    by the map $M\mapsto[M]$).  Since $\phi(\Gamma_0)\subset
    SL_n(\ZZ)$, we clearly have that the map $X_0\rightarrow
    F(\mathcal M^*(n,p,n\ord-k))$ given by
    \[x\mapsto(SL_n\ZZ_p)f(x)
    \]
    is $\Gamma_0$-invariant.  By ergodicity of $\Gamma_0\actson X_0$,
    we may suppose that $f(X_0)$ is contained in a fixed
    $SL_n(\ZZ_p)$-orbit.  Hence, the claim follows from Proposition
    \ref{prop_type_preserve}.
  \end{claimproof}
  
  The proof now concludes in a manner similar to that of Theorem
  \ref{thm_Bpart1'}.  Roughly speaking, we shall choose generic
  $x,gx\in X_0$ so as to guarantee that $f(gx)=g^{-T}f(x)$.  We shall
  also select $g\in GL_n(\QQ)$ so that $g^{-T}$ fails to preserve
  lattice types on $f(X_0)$, contradicting the last claim.

  More specifically, let $g=\diag(1/p,1,\ldots,1)$, where
  $\diag(d_1,\ldots,d_n)$ denotes the diagonal matrix with
  $a_{ii}=d_i$.  Arguing as in the proof of Theorem \ref{thm_Bpart1},
  we can choose $x\in X_0$ such that $gx\in X_0$ and $x^\bot\in
  Fr(PGL_n(\QQ)\actson Gr_{n\ord-k}\QQ_p^n)$.  Now, both $x$ and $gx$
  have the canonical complementary subspace
  $V_1=\QQ_pe_{k+1}\oplus\cdots\oplus\QQ_pe_n$.  Since $g$ acts on
  $V_1$ by the identity map, we clearly have $g(x,L_0)=(gx,L_0)$, and
  so $g$ witnesses that $x\mathrel{E}_\oiso^kgx$.  It follows that
  \mbox{$f(x)(\iso_{n,p}^{n\ord-k})^*f(gx)$}, and hence there exists
  $h\in GL_n(\QQ)$ such that $hf(x)=f(gx)$.  Then $h\bar f(x)=\bar
  f(gx)$, so now
  \[hx^\bot=h\bar f(x)=\bar f(gx)=(gx)^\bot=g^{-T}x^\bot.
  \]
  Since we have chosen $x^\bot$ so that it is not fixed by any element
  of $PGL_n(\QQ)\backslash\{1\}$, there exists $a\in\QQ_p^*$ such that
  $h=ag^{-T}=\diag(ap,a,\ldots,a)$ (see Figure \ref{fig_witness}
  below).

  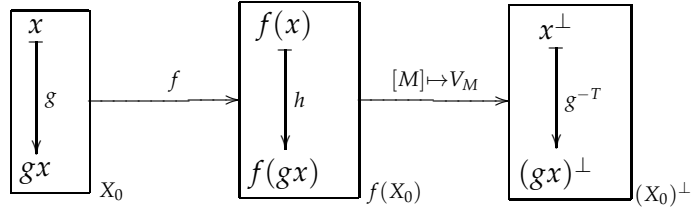
\begin{figure}[h]
  \begin{displaymath}
  \def\g#1{\save[].[d]!C="g#1"*[F]\frm{}\restore}
  \xymatrix @C=2cm @R=1.3cm {
    \g1 x\quad\ar@<-1ex>@{|->}[d]^g &
    \g2 f(x)\quad\ar@<-1ex>@{|->}[d]^h &
    \g3 x^\bot\quad\ar@<-1ex>@{|->}[d]^{g^{-T}}\\
    gx\quad \ar@{}[r]_</.3cm/{X_0} &
    f(gx)\quad \ar@{}[r]_</.48cm/{f(X_0)} &
    (gx)^\bot\quad \ar@<0cm>@{}[l]_</-2.05cm/{(X_0)^\bot}
    \ar "g1";"g2"^f
    \ar "g2";"g3"^{[M]\mapsto V_M}
  }
  \end{displaymath}
  \caption{We have guaranteed that $f(gx)=hf(x)$, where
    $h=ag^{-T}$.\label{fig_witness}}
  \end{figure}

  Finally, $x^\bot,(gx)^\bot$ each have the canonical complementary
  subspace $V_0=\QQ_pe_1\oplus\cdots\oplus\QQ_pe_k$.  Letting $h^c$
  denote the \emph{upper left}-hand corner of $h$, we have
  $\nu_p(h^c)\equiv 1\mod k$, and so by Proposition \ref{prop_type},
  $h^c$ acts in a type-\emph{altering} fashion on $\mathcal L(V_0)$.
  But we have arranged for $\type(f(x))=\type(f(gx))$, a
  contradiction.
\end{proof}

The second part of Theorem B follows immediately.

\begin{thm}[Theorem B, part 2]
  If $n\geq3$ and $k\leq n-2$, then
  $E_{GL_n\QQ}^k\;\not\leq_B^w\;E_\oiso^k$.
\end{thm}

\begin{proof}
  The orthogonal complement map witnesses that $E_{GL_n\QQ}^k$ is
  Borel bireducible with $E_{GL_n\QQ}^{n\ord-k}$.  Hence, if
  $E_{GL_n\QQ}^k\leq_B^wE_\oiso^k$ then there exists a weak Borel
  reduction $f$ from $E_{GL_n\QQ}^{n-k}$ to $E_\oiso^k$.  Clearly, $f$
  is also a weak Borel reduction from $E_\oiso^{n-k}$ to $E_\oiso^k$,
  contradicting Theorem \ref{thm_Acase2'}.
\end{proof}

The keen-eyed reader will have noticed that Theorem A is as yet
incomplete.

\begin{thm}[Theorem A, case 3]
  \label{thm_Acase3}
  If $n\geq3$ then $E_\oiso^1\not\leq_BE_\oiso^{n\ord-1}$.
\end{thm}

\begin{proof}
  By Theorem 4.4 of \cite{plocal}, for groups $A,B\in R(n,p,n\ord-1)$
  we have that $A$ is quasi-isomorphic to $B$ iff $A$ is isomorphic to
  $B$.  In particular, $\oiso_{n,p}^{n\ord-1}$ is Borel bireducible
  with $\oqiso_{n,p}^{n\ord-1}$, and it follows that
  $E_{\oiso}^{n\ord-1}$ is Borel bireducible with
  $E_{GL_n\QQ}^{n\ord-1}$.  Again using the orthogonal complement map,
  $E_{GL_n\QQ}^{n\ord-1}$ is Borel bireducible with $E_{GL_n\QQ}^1$,
  and so we have established that the right-hand side
  $E_\oiso^{n\ord-1}$ is Borel bireducible with $E_{GL_n\QQ}^1$.
  Hence, the result follows from Theorem B, part 1.
\end{proof}

\bibliographystyle{alphanum}
\begin{singlespace}
  \bibliography{cber}
\end{singlespace}

\end{document}